\documentclass[a4paper]{article}
\usepackage{amsthm,amsfonts,amsmath,amssymb,units}
\usepackage[abbrev,nobysame]{amsrefs}
\usepackage[cp1251]{inputenc}
\usepackage[english]{babel}
\usepackage[final]{graphicx}
\usepackage{setspace}
\usepackage[12pt]{extsizes}
\begin{document}
\newcommand{\per}{{\rm per}}
\newcommand{\supp}{{\rm supp}}
\newtheorem{theorem}{Theorem}
\newtheorem{lemma}{Lemma}
\newtheorem{utv}{Proposition}
\newtheorem{svoistvo}{Property}
\newtheorem{sled}{Corollary}
\newtheorem{con}{Conjecture}
\newtheorem{zam}{Remark}
\newtheorem{quest}{Question}
\newtheorem{claim}{Claim}

\author{Anna A. Taranenko\thanks{Sobolev Institute of Mathematics,  Novosibirsk, Russia. Email: \texttt{taa@math.nsc.ru}}}
\title{Multidimensional quadrangle condition and cuboctahedra in latin hypercubes}
\date{April 1, 2026}

\maketitle

\begin{abstract}
The well-known quadrangle criterion  states that  a  latin square is  isotopic to  the  Cayley table of  a group if and only if all quadrangles spanned by the same triple of symbols coincide on the fourth symbol.  Gowers and Long (2020) reformulated this result in the following way: the Cayley tables of the most associative quasigroups  have the  maximum  number of octahedra.

In the present paper, we state the multidimensional quadrangle condition for $d$-dimensional latin hypercubes in terms of the reconstruction of submatrices of order $2$ from a bundle of $d+1$ entries and in terms of the maximal number of cuboctahedra. In particular, we show that  the most associative $d$-ary quasigroups  have Cayley tables such that  every $2$-dimensional plane is isotopic to a latin square that is principally isotopic to the Cayley table of a group. We also estimate the number of cuboctahedra in latin squares and hypercubes from below and provide  computational results.


\end{abstract}

\section{Introduction}

A \textit{latin square} $L$ of order $n$ is an $n \times n$-array filled with $n$ symbols so that each row and each column of $L$ contains all $n$ symbols. Every latin square can be considered as the Cayley (multiplication) table of a quasigroup: a binary operation with the left and right division.  

There  is a natural question about the conditions under which a latin square is the multiplication table of a group, i.e., a quasigroup with associativity  and an identity element. Permutations of the rows, columns and symbols of a latin square  allow us to obtain another (isotopic) latin square that is the Cayley table of a quasigroup with an identity element.  About  a century ago, Brandt~\cite{brandt.qudrangle}  and  Frolov~\cite{frolov.groupCayley}  proposed a  simple criterion for associative latin squares that is now known as the \textit{quadrangle condition}.

\begin{theorem}[\cite{brandt.qudrangle,frolov.groupCayley}] \label{BrandtFrolovthm}
A latin square $L$ is isotopic to the Cayley table of a group if and only if for all $a,b,c$ forming a  submatrix
$$
\begin{array}{cc}
a & b \\ c & d
\end{array}
$$
the entry $d$ is uniquely defined.
\end{theorem}

In~\cite{GowLong.approxgroups}, Gowers and Long raised the question of how associative a quasigroup  can be. As one possible measure of the degree of associativity, they proposed counting the number of pairs of identical $2 \times 2$ submatrices in a latin square, which they called ``octahedra''. Latin squares of order $n$  isotopic to the Cayley table of a group have $n^5$ ``octahedra'', which is the maximum possible number, and a larger number of ``octahedra''  corresponds to greater  associativity of a latin square.  

In the present paper, we use the term ``cuboctahedron'' instead of ``octahedron''. It was introduced in~\cite{KwanSahSaw.SubstrctLS}  by Kwan, Sah, Sawhney, and Simkin and refers to the  geometric  resemblance of this object  to a cuboctahedron in the corresponding $3$-partite $3$-uniform hypergraph.  In~\cite{KwanSahSaw.SubstrctLS},  it  was  shown that a random latin square of order $n$ has $ (4 + o (1))n^4$ cuboctahedra with high probability.

The main aim of the present paper is to state the quadrangle condition for latin hypercubes and multiary quasigroups.   In Section~\ref{defsec}, we introduce the multidimensional analogue of a cuboctahedron,  give the necessary definitions, and provide preliminary results. Section~\ref{maxnumbersec} is devoted to the proof of the multidimensional  quadrangle condition.  We  show that the Cayley tables of the  most associative  $d$-ary quasigroups  have  $2$-dimensional planes  isotopic to latin squares that are principally isotopic to the Cayley tables of groups.  As in the $2$-dimensional case, we also state the multidimensional quadrangle condition  in terms of  the reconstruction of  submatrices of order $2$ from a bundle of  entries and    the maximality of the number of cuboctahedra in a latin hypercube.

In addition, in Section~\ref{lowboundsec} we obtain several lower bounds on the number of cuboctahedra in latin squares and hypercubes.   In Section~\ref{enumersec}, we compare these bounds with the actual numbers of cuboctahedra in latin squares of orders $n \leq 6$ and  in $3$-dimensional and $4$-dimensional latin hypercubes of orders $4$ and $5$.

\section{Definitions and preliminaries} \label{defsec}

 A \textit{$d$-ary quasigroup $f$ of order $n$} is a $d$-ary operation $f:  \{ 0, \ldots, n-1 \}^d \rightarrow \{ 0, \ldots, n-1 \}$ such that the equation $x_0 = f(x_1, \ldots, x_d)$ has a unique solution for any one variable if all the other $d$ variables are specified arbitrarily.  A \textit{$d$-dimensional latin hypercube} $Q$ is the Cayley table of  a $d$-ary quasigroup of the same order.

A \textit{composition}  of a $d$-ary quasigroup $f$ and a $k$-ary quasigroup $g$ of order $n$  is the $(d+k-1)$-ary quasigroup  $h$ of order $n$ such that for some permutation $\sigma \in S_{d+k}$ it holds
$$h(x_1, \ldots, x_{d+k-1}) = x_{d+k} \Leftrightarrow g(x_{\sigma(1)}, \ldots, x_{\sigma(k)}) = f(x_{\sigma(k+1)},   \ldots, x_{\sigma(d+k)} ). $$
Given a binary quasigroup $f$  of order $n$,  the \textit{$d$-iterated quasigroup} is the $(d+1)$-ary quasigroup of order $n$ defined by the relation $x_0 = f( \cdots (f (f(x_1, x_2), x_3), \cdots, x_{d})$ 

Given $d$-dimensional  latin hypercube $Q$ of order $n$, corresponding to a $d$-ary quasigroup $f$,  $k\in \left\{0,\ldots,d\right\}$ and $1 \leq i_1 < \cdots < i_k \leq d$,  a \textit{$k$-dimensional plane of direction $(i_1, \ldots, i_k)$}  in $Q$ is a $k$-dimensional  latin subcube of the same order  obtained by   letting   variables $x_{i_1}, \ldots, x_{i_k}$ of $f$  vary from $0$ to $n-1$ and  fixing  the values of  the remaining $d-k$ variables of $f$.

A $1$-dimensional plane in $Q$ is said to be a \textit{line}, and a $(d-1)$-dimensional plane is a \textit{hyperplane}.    We will say that hyperplanes are \textit{parallel} if they have the same direction.

In what follows, $d$-dimensional latin hypercubes  $Q$ of order $n$  are also considered as  arrays $(q_\alpha)_{\alpha}$, whose entries are indexed by tuples  $  \alpha = (\alpha_1, \ldots, \alpha_d ),  \alpha_i \in \{ 0, \ldots, n-1\}$,  that are  filled by $n$ symbols so that in each line all symbols are different.

Latin hypercubes  $Q$ and $Q'$ are called \textit{isotopic} if one can be obtained from the other by permutations of  symbols and parallel hyperplanes. If $Q$ is obtained from $Q'$ only  by permutations of parallel hyperplanes, then $Q$ is \textit{principally isotopic} to $Q'$.

$d$-dimensional latin hypercubes $Q$ and $Q'$ corresponding to quasigroups $f$ and $f'$ are said to be \textit{conjugate} if  there exists a permutation $\pi \in S_{d+1}$ such that $x_0 = f(x_1, \ldots, x_d)$ if and only if $x_{\pi(0)} = f' (x_{\pi(1)}, \ldots, x_{\pi(d)})$.    Latin   hypercubes  $Q$ and  $Q'$   belong to the same \textit{main class}  if $Q$ is isotopic to any conjugacy of $Q'$.

Now we are ready to give the key definition of the paper. 
A \textit{cuboctahedron} in a $d$-dimensional  latin hypercube $Q$ is a collection  of $d$ pairs $(a_1^{i}, a_2^{i})$ and a   collection  of $d$ pairs $(b_1^{i}, b_2^{i})$, $i = 1, \ldots, d$ such that for all $j_1, \ldots, j_d \in \{1,2\}$ we have $q_{a_{j_1}^1, \ldots, a_{j_d}^d } = q_{b_{j_1}^1, \ldots, b_{j_d}^d } $. In other words, a cuboctahedron is a  pair of submatrices $\{ q_{a_{j_1}^1, \ldots, a_{j_d}^d } \}_{j_1, \ldots, j_d \in \{1,2\}}$  and   $\{ q_{b_{j_1}^1, \ldots, b_{j_d}^d } \}_{j_1, \ldots, j_d \in \{1,2\}}$   of order $2$,  (possibly degenerate and having width $1$ in some directions)   that have exactly the same  pattern of symbols.  We call entries   $q_{a_{1}^1, \ldots, a_{1}^d } $ and $q_{b_{1}^1, \ldots, b_{1}^d }$   of these submatrices    \textit{corners}.

We will say that a cuboctahedron  $\{ (a_1^{i}, a_2^{i}), (b_1^{i}, b_2^{i}) : i = 1, \ldots, d \}$ in a $d$-dimensional latin hypercube has the \textit{dimension} $k$ if there are exactly $k$ pairs $(a_1^{i}, a_2^{i})$ (and $(b_1^{i}, b_2^{i})$) such that $a_1^{i} \neq a_2^{i} $ (and $b_1^{i} \neq b_2^{i}$). In particular, a $0$-dimensional cuboctahedron is a pair of entries containing the same symbols, and a $1$-dimensional cuboctahedron is  a pair of dominos filled in the same way.

From  the definitions  it easily follows   that isotopic latin hypercubes have the same number of cuboctahedra of any given dimension.  Meanwhile, a conjugation of a latin hypercube may change its number of cuboctahedra. For example, we verified that every $4$-dimensional latin hypercube of order $5$, except for those from one of the main classes, has a conjugate latin hypercube with a different number of cuboctahedra.

As noted in~\cite{GowLong.approxgroups},  Theorem~\ref{BrandtFrolovthm} implies  that every Latin square  of order $n$ contains at most $n^5$ cuboctahedra and that a  latin square  is  isotopic to the Cayley table of a group if and only if it contains exactly $n^5$ cuboctahedra. We reformulate Theorem~\ref{BrandtFrolovthm} in the following form, which will be used later.

\begin{utv} \label{cubisotop}
Let $L$ and $L'$ be latin squares of order $n$ such that  for all triples of symbols $a,b,c$, $a \notin \{b,c\}$, the  $*$-entry in the submatrix
$$
\begin{array}{cc}
a & b \\ c & *
\end{array}
$$
in $L$ and $L'$ is  the same symbol $d$. Then  $L$ and $L'$ are principally isotopic  and they  are  isotopic to a Cayley table of a group. 
\end{utv}

\begin{proof}
By Theorem~\ref{BrandtFrolovthm}, the latin squares $L$ and $L'$ are isotopic to the Cayley tables of groups. To show that they are principally isotopic, it is sufficient  to permute rows and columns of $L$ and $L'$  so that  they have the same first row and the same first column. The condition on the submatrices implies that the resulting  latin squares are identical. 
\end{proof}

\section{Multidimensional quadrangle condition} \label{maxnumbersec}

We start this section with a generalization of the quadrangle condition to the multidimensional case. For this purpose, we need additional definitions.

Consider  a $d$-dimensional latin hypercube $Q$ of order $n$.  A \textit{bundle} is a tuple of  indices $( \alpha^0, \alpha^1, \ldots, \alpha^d)$ of $d+1$ entries of $Q$ such that for each $i \in \{ 1, \ldots, d\}$  indices $\alpha^0$ and $\alpha^i$ belong to the line of direction $(i)$.   A $d$-dimensional  submatrix $S$ of order $2$ is \textit{spanned} by a bundle $( \alpha^0, \alpha^1, \ldots, \alpha^d)$ if it is composed by indices $\beta$    such that if for some $i \in \{ 1, \ldots, d\}$ we have $\beta_i \neq \alpha^0_i$, then $\beta_i = \alpha^i_i$. The filling  of a bundle or other  submatrix  $S$ in $Q$   is an indexed collection of symbols  $\{q_\beta \}_{\beta \in S}$.

\begin{theorem} \label{uniqrestor}
Let $Q$ be a $d$-dimensional latin hypercube of order $n$.  The   filling of a bundle uniquely determines the filling   of a spanned submatrix of order $2$ if and only if  for every direction $(i,j)$ there exists  a latin square $L^{i,j}$ isotopic to a Cayley table of a group  such that  every $2$-dimensional plane of $Q$ of direction $(i,j)$ is principally isotopic to $L^{i,j}$. 
\end{theorem}

\begin{proof}
The proof is by induction on $d$. For $d = 2$ the statement is equivalent to Theorem~\ref{BrandtFrolovthm}. 

Assume   that  the filling of  every  bundle  in $Q$  uniquely determines the filling   of the spanned submatrix  of order $2$. Let us fix some direction $(i,j)$ of $2$-dimensional planes and consider a  hyperplane $\Gamma$  that either contains $2$-dimensional planes of this direction or does not intersect them. 

Since  $\Gamma$ is a $(d-1)$-dimensional latin hypercube, the filling of  every  bundle  in  $\Gamma$  uniquely determines the filling   of the spanned $(d-1)$-dimensional submatrix.    By the inductive assumption,  all  $2$-dimensional planes of direction $(i,j)$ in $\Gamma$ are principally isotopic to a latin square $L$ which is isotopic to the Cayley table of some group. 

Similarly, consider a hyperplane $\Gamma'$ parallel to hyperplane $\Gamma$.  Then all  $2$-dimensional planes of direction $(i,j)$ in $\Gamma'$ are principally isotopic to a (possibly different) latin square $L'$ which is isotopic to the Cayley table of some group.  In all  $2$-dimensional planes of direction $(i,j)$ in $\Gamma$ and $\Gamma'$   submatrices of order $2$  are defined uniquely  by three other symbols. Consequently, from  Proposition~\ref{cubisotop} we have  that  latin squares $L$ and $L'$ are principally isotopic. 

Assume now that  for every direction $(i,j)$ there is  a latin square $L^{i,j}$ which is isotopic to the Cayley table of some group such that every $2$-dimensional plane of direction $(i,j)$   in $Q$ is principally isotopic to  $L^{i,j}$. Consider a $d$-dimensional submatrix $S$ of order $2$  spanned  by a bundle  $(\alpha^0, \alpha^1, \ldots, \alpha^d)$ and  filled by symbols $q^0,  q^1, \ldots, q^d $.  By the inductive assumption, we uniquely determine the symbols  in  all $d$ hyperplanes of the submatrix $S$ that contain the index $\alpha^0$.  Using these symbols, we uniquely determine the filling of hyperplanes of $S$ going through indices $\alpha^1, \ldots, \alpha^d$. Thus we completely reconstruct the filling of the submatrix $S$.
\end{proof}

Our next goal is to  state the quadrangle condition of Theorem~\ref{uniqrestor} in terms of the maximality of the number of cuboctahedra in  a latin hypercube. We start with  an upper bound on the number of $k$-dimensional cuboctahedra.

\begin{utv} \label{kdimupper}
Given $0 \leq k \leq d$, every  $d$-dimensional latin hypercube $Q$ of order $n$ has no more than ${d \choose k}  n^{2d-1}   (n-1)^{k} $  $k$-dimensional cuboctahedra and $n^{3d-1}$  cuboctahedra  in total. 
\end{utv}

\begin{proof}
Recall that a $k$-dimensional cuboctahedron is a collection of pairs $\{ (a_1^{i}, a_2^{i}), (b_1^{i}, b_2^{i}) : i = 1, \ldots, d \}$ such that there are exactly $k$ pairs $(a_1^{i}, a_2^{i})$ for which $a_1^{i} \neq a_2^{i} $.   We have  ${d \choose k}   n^{d-k}   (n(n-1))^{k}$ ways to choose  $d$ pairs $(a_1^{i}, a_2^{i})$ such that there are exactly   $k$ inequalities   $a_1^{i} \neq  a_2^{i}$.   The condition $q_{a_{1}^1, \ldots, a_{1}^d } = q_{b_{1}^1, \ldots, b_{1}^d } $ allows us to choose arbitrarily all except one  $b_{1}^1, \ldots, b_{1}^d$   that can be done in $n^{d -1}$ ways. The remaining values of $b_{2}^1, \ldots, b_{2}^d$ for a $k$-dimensional cuboctahedron are uniquely defined from equalities $q_{a_{j_1}^1, \ldots, a_{j_d}^d } = q_{b_{j_1}^1, \ldots, b_{j_d}^d } $ for all $j_1, \ldots, j_d \in \{ 1,2\}$. Thus the maximum number of $k$-dimensional cuboctahedra in a $d$-dimensional latin hypercube of order $n$ is ${d \choose k}  n^{2d-1}   (n-1)^{k}  $.

To find the total number of cuboctahedra, it is sufficient to sum these values from $0$ to $d$. 
\end{proof}

Now we are ready to prove the second form of the quadrangle condition.

\begin{theorem} \label{quadmaxthm}
A $d$-dimensional latin hypercube $Q$ of order $n$ has  $n^{3d-1}$ cuboctahedra if and only if  every $2$-dimensional plane of $Q$ of direction $(i,j)$ is principally isotopic to a latin square $L^{i,j}$ which is isotopic to a Cayley table of some group. 
\end{theorem}

\begin{proof}
By Proposition~\ref{kdimupper}, $n^{3d-1}$  is the maximum number of cuboctahedra in a  $d$-dimensional latin hypercube of order $n$. 

A latin hypercube $Q$ has the maximum number of cuboctahedra if every $d$-dimensional submatrix $S$ of order $2$ spanned by a bundle  $(\alpha^0, \alpha^1, \ldots, \alpha^d)$  and filled by symbols $q^0, q^1, \ldots, q^d$     coincides with a submatrix $S'$  spanned by a  bundle  $(\beta^0, \beta^1, \ldots, \beta^d)$ and  filled by the same symbols. By Theorem~\ref{uniqrestor}, this condition is equivalent to the fact that  every $2$-dimensional plane of $Q$ of direction $(i,j)$ is principally isotopic to a latin square $L^{i,j}$ which is isotopic to a Cayley table of some group. 
\end{proof}

We conclude this section with some comments on  the obtained characterization of latin hypercubes with the maximum number of cuboctahedra. First, Theorem~\ref{quadmaxthm} implies that the Cayley tables of $d$-iterated groups $G$ are the most associative and have the maximal number of cuboctahedra because every $2$-dimensional plane of such a latin hypercube is principally isotopic to the group $G$. In Section~\ref{enumersec}, we find examples of latin hypercubes with the maximal number of cuboctahedra that are not isotopic to iterated groups.  

Next, in~\cite[Corollary 2]{KrotPot.redquasi} it was proved that  latin hypercubes of order $5$ and dimension $d \geq 4$ and  latin hypercubes of order $7$ and dimension $d \geq 3$ in which every $2$-dimensional plane is isotopic to a group correspond to reducible quasigroups. Taking into account the enumeration results for $3$-dimensional latin hypercubes of order $5$ (see Section~\ref{enumersec}), we deduce that every $d$-dimensional latin hypercube of order $5$ and $7$, $ d \geq 3$, containing the maximum number of cuboctahedra, is reducible.  Meanwhile, there exists a $4$-dimensional latin hypercube of order $4$ (hypercube \#23 from the computational data for~\cite{KayWan.census}) with the maximum number of cuboctahedra that is not reducible.  We conjecture that if $n$ is prime, then  every $d$-dimensional latin hypercube of order $n$ with the maximum number of cuboctahedra is reducible.

\section{Lower bounds on the number of cuboctahedra} \label{lowboundsec}

In this section, we prove several lower bounds on the number of cuboctahedra in latin squares and hypercubes.  We start with  cuboctahedra of small dimensions. 

\begin{utv} \label{0dimcub}
The number $c_0 (Q)$ of $0$-dimensional cuboctahedra in a  $d$-dimensional latin hypercube $Q$ of order $n$ is $c_0(Q) = n^{2d - 1}$.
\end{utv}

\begin{proof}  
A $0$-dimensional cuboctahedron is an ordered pair of entries with identical symbols in a latin hypercube.  For each of $n$ symbols, there are $n^{2d-2}$ ordered pairs of entries of $Q$ containing a given symbol.
\end{proof}

\begin{utv} \label{1dimcub}
The number $c_1 (Q)$ of $1$-dimensional cuboctahedra in a  $d$-dimensional latin hypercube $Q$ of order $n$ is $c_1(Q) =  d n^{2d-1} (n-1)$.
\end{utv}

\begin{proof}
Note that a $1$-dimensional cuboctahedron is an ordered pair of identically filled ordered dominos (pairs of entries from one line). The number of ways to choose the first domino is   $d n^{d} (n-1)$: there are $d$ ways to choose the direction of the domino, $n^{d}$ ways to choose the first entry, and $(n-1)$ ways to choose the second entry in the line of this direction. The number of ways to choose the second domino is $n^{d-1}$ because it is uniquely defined by its first entry.  
\end{proof}

Next, we find the number of  degenerate cuboctahedra.  A cuboctahedron  $\{ (a_1^{i}, a_2^{i}), (b_1^{i}, b_2^{i}) : i = 1, \ldots, d \}$ in a $d$-dimensional latin hypercube  $Q$ is \textit{degenerate} if $a_j^i = b_j^i$ for every $ i \in  \{1, \ldots, d\}$ and $j \in \{ 1,2\}$. In other words, a degenerate cuboctahedron in $Q$ is a submatrix, taken twice. 

\begin{utv} \label{degenercub}
The number $\delta_k(Q)$ of degenerate $k$-dimensional cuboctahedra in a $d$-dimensional latin hypercube $Q$ of order $n$ is $\delta_k(Q) =  {d \choose k} n^{d} (n-1)^k$.
\end{utv}

\begin{proof}
By definition, the number  of degenerate $k$-dimensional cuboctahedra is equal to the number of $k$-dimensional submatrices of order $2$ in $Q$.  There are ${d \choose k}$ ways  to choose the direction of the submatrix, $n^{d}$ ways to choose the first entry, and $(n-1)^k$ ways to choose other entries  in selected $k$  lines. So there are  ${d \choose k} n^{d} (n-1)^k$ degenerate $k$-dimensional cuboctahedra in total. 
\end{proof}

The above statements imply that the total number of cuboctahedra in a given latin hypercube is determined by the number of non-degenerate cuboctahedra of dimensions $k \geq 2$. For latin squares, we give a better bound on the number of non-degenerate $2$-dimensional cuboctahedra (and hence on the total number of cuboctahedra), while for higher-dimensional latin hypercubes we give, for the sake of simplicity, only the essential parts of the lower bound.

\begin{theorem} \label{lowboundsquare}
Every latin square  $L$ of order $n$ contains at least $3 n^4 + n^3 - 9 n^2 + 6n$ cuboctahedra. 
\end{theorem}

\begin{proof}
By Propositions~\ref{0dimcub} and~\ref{1dimcub}, there are $2n^4 - n^3$ cuboctahedra of dimensions $0$ and $1$  in every latin square of order $n$.  By Proposition~\ref{degenercub}, we have in addition $n^4 - 2n^3 + n^2$ degenerate cuboctahedra of dimension $2$. 

Let us estimate the number of non-degenerate $2$-dimensional cuboctahedra. 
Given symbols $x,y,z \in \{ 0, \ldots, n-1\}$,  let  $t^u_{x,y,z}$ denote the number  of $2 \times 2$ submatrices $S^u_{x,y,z}$ of the form
$$
\begin{array}{cc}
 x & y \\ 
 z & u
\end{array}
$$
 in the latin square $L$.   Since every pair of submatrices  $S^u_{x,y,z}$  in $L$ is a non-degenerate $2$-dimensional cuboctahedron, the total number of non-degenerate $2$-dimensional cuboctahedra is  equal to $\sum\limits_{x,y,z,u} t^u_{x,y,z} (t^u_{x,y,z} - 1) $. 
 From the definition of a latin square,  for all $x,y,z$, where  $x \neq y$, $x \neq z$, we have  that $\sum\limits_{u=0}^{n-1} t^u_{x,y,z} = n $ and  $ t^y_{x,y,z} =  t^z_{x,y,z} = 0$.
 
Consider the case $y \neq z$. By the pigeonhole principle, either  there exists $u$ such that $ t^u_{x,y,z}  \geq 3 $ or there exist two different $u$ and $u'$ such that $ t^u_{x,y,z}  \geq 2$ and $ t^{u'}_{x,y,z}  \geq 2$. In the first case, three submatrices $S^u_{x,y,z}$  form $6$  non-degenerate $2$-dimensional cuboctahedra, and in the second case two pairs of $S^u_{x,y,z}$ and $S^{u'}_{x,y,z}$ give  $4$  non-degenerate $2$-dimensional cuboctahedra. Assuming that for all  $x,y,z$, where $x \neq y$, $x \neq z$, $y \neq z$, the second  (worst) case is realized, we obtain at least $4 n (n-1) (n-2) = 4 n^3 - 12n^2 + 8n$  non-degenerate $2$-dimensional  cuboctahedra  of this type in total. 
 
Now assume that $y = z$. Again, by the pigeonhole principle, we see that  there exists $u$ such that $ t^u_{x,y,y}  \geq 2 $.  Since each intercalate $S^x_{x,y,y}$ is counted twice, the number $t^x_{x,y,y}$ is always even.  So in this case the minimum number of cuboctahedra is achieved when there are exactly two submatrices  $S^x_{x,y,y}$  for each $x,y$, where $x \neq y$, and all other submatrices  $S^u_{x,y,y}$ appear once. A pair of submatrices  $S^x_{x,y,y}$   gives us $2$ non-degenerate $2$-dimensional cuboctahedra, and so in total we have $2 n (n-1) = 2n^2 - 2n$ additional cuboctahedra. 
 
 Summing up, every latin square $L$ of order $n$ has at least
 $$(2n^4 - n^3) + (n^4 - 2n^3 + n^2) + (4 n^3 - 12n^2 + 8n) + (2n^2 - 2n) = 3n^4 + n^3 - 9 n^2 + 6n$$
 cuboctahedra.

\end{proof}

Using a similar method, we estimate from below the number of non-degenerate $k$-dimensional cuboctahedra in latin hypercubes.

\begin{utv} \label{kdimnondeg}
Let  $k$, $n$, and $d$ be integers,  $n \geq 2$,  $d \geq  3$,  and $2 \leq k \leq \min \{d ,n\}$.
Then  every $d$-dimensional   latin hypercube  $Q$ of order $n$ contains at least $ {d \choose k} n (n-k)^{2^k - 1} r_k(r_k-1)$ non-degenerate $k$-dimensional cuboctahedra, where $r_k = \lfloor  n^{d-1} (n-1)^{k + 1 - 2^k}  \rfloor$.
\end{utv}

\begin{proof}
Fix a symbol $s \in \{ 0,  \ldots, n-1\}$ and one of  ${d \choose k}$ directions of $k$-dimensional planes in $Q$. Let us estimate the number of non-degenerate $k$-dimensional cuboctahedra  of a chosen direction with a corner symbol $s$. 

By the definition, each cuboctahedron is a pair of $k$-dimensional submatrices of order $2$ in $Q$ with the same pattern of symbols.  So if there are $t$ different $k$-dimensional submatrices of order $2$ with the fixed corner and a giving filling by symbols, then they contribute $t(t-1)$   non-degenerate $k$-dimensional cuboctahedra. 

The number of $k$-dimensional  submatrices of order $2$ of a given direction  and  having  a fixed entry as a corner is $(n-1)^k$, and the total number of such submatrices with corner symbol $s$ is $ n^{d-1} (n-1)^k$. 

On the other hand, the number of possible  symbol patterns of $k$-dimensional matrices of order $2$  in a latin hypercube of order $n$ with a fixed corner symbol $s$   does not exceed $(n-1)^{2^k -1}$ and is not less than $(n-k)^{2^k-1}$.  The minimum of the number of cuboctahedra is achieved when each possible symbol pattern of $k$-dimensional submatrices of order $2$ appears an equal number of times as possible, i.e., when it appears at least  $\lfloor  n^{d-1} (n-1)^{k + 1 - 2^k}  \rfloor = r_k$ times.   So there are at least $(n-k)^{2^k-1} r_k(r_k-1)$ non-degenerate  $k$-dimensional  cuboctahedra of a given direction and with the corner  symbol $s$, and $ {d \choose k} n (n-k)^{2^k-1} r_k(r_k-1)$  cuboctahedra in total.

\end{proof}

Note that for small $n$ the bounds $(n-1)^{2^k-1}$ and $(n - d)^{2^k-1}$ on the number of possible symbol patterns  in $k$-dimensional matrices of order $2$ 
 can be quite far from the exact number. So  in Section~\ref{enumersec}  we estimate them more precisely  for latin hypercubes of small order   and dimension.
 
 At last, we are ready to state the lower bound on the number of cuboctahedra in latin hypercubes.  

\begin{theorem} \label{cublow}
Let $n$ and $d$ be integers,  $d \geq  3$ and $n \geq d$. For integer $2 \leq k \leq d$, define $r_k = \lfloor n^{d-1} (n-1)^{k + 1 - 2^k}  \rfloor$.  Then every $d$-dimensional   latin hypercube  $Q$ of order $n$ contains at least 
$$(d+1) n^{2d} - (d-1) ( n^{2d-1} - n^{d} )  - d n^{d+1}   + n\sum\limits_{k = 2}^d   {d \choose k}  (n-k)^{2^k-1} r_k(r_k-1) $$
  cuboctahedra. 
\end{theorem}

\begin{proof}
By Propositions~\ref{0dimcub} and~\ref{1dimcub}, there are $dn^{2d} - (d-1)n^{2d-1} $ cuboctahedra of dimensions $0$ and $1$ in a $d$-dimensional  latin hypercube of order $n$.  
By Proposition~\ref{degenercub}, we have in addition $\sum\limits_{k = 2}^d  {d \choose k} n^d (n-1)^k = n^d (n^d - d (n-1) - 1)$ degenerate cuboctahedra of dimension $k$. Finally,  for each $k \geq 2$, Proposition~\ref{kdimnondeg}  gives $ {d \choose k}  n (n-k)^{2^k -1} r_k(r_k - 1)$ $k$-dimensional non-degenerate cuboctahedra. 
\end{proof}

\begin{sled} 
For $d \geq 3$, every $d$-dimensional   latin hypercube  $Q$ of order $n$ contains at least 
 $$\left(\frac{d(d-1)}{2} + d +1 \right) n^{2d} + o(n^{2d})$$ 
 cuboctahedra as $n \rightarrow \infty$.
\end{sled}

\begin{proof}
In Theorem~\ref{cublow}  we have $r_2 = \lfloor \frac{ n^{d-1}}{n-1}    \rfloor $, therefore the corresponding  summand    gives asymptotically $\frac{d(d-1)}{2} n^{2d}$   non-degenerate  $2$-dimensional cuboctahedra. All other summands for $k \geq 3$  in Theorem~\ref{cublow}  are $o (n^{2d})$. 
\end{proof}

\section{Enumeration results} \label{enumersec}

In this section, we provide several enumeration results on the numbers of cuboctahedra in latin hypercubes of small order and dimension.  They are based on the collection of representatives of all main classes of small latin squares by McKay~\cite{mckay.LSsite} and  small latin hypercubes by McKay and Wanless~\cite{KayWan.census}.   For each  $d$-dimensional latin hypercube $Q$ of order $n$ and $0 \leq k \leq d$, we use exhaustive enumeration to count the number of $k$-dimensional cuboctahedra  in $Q$ and  in  all $d$ conjugates of $Q$ obtained by  the transposition of the symbol position and the $i$-th position in indices,  $i = 1, \ldots, d$. The Python code for enumerating cuboctahedra in a given latin hypercube can be found at~\cite{cuboctahedra.code}.

Below, we present the maximum and minimum numbers of cuboctahedra in latin hypercubes of small order and dimension, along with the hypercubes that achieve these numbers.  Moreover, we compare the minimum numbers of cuboctahedra with the theoretical lower bounds proved in Section~\ref{lowboundsec}. We start with latin squares, i.e., $2$-dimensional latin hypercubes.

\textbf{Latin squares of order $n \leq 4$} are isotopic to the Cayley tables of groups and have $n^5$ cuboctahedra. 

\textbf{Latin squares of order $5$}.
There are only two isotopy classes of latin squares of order $5$: one of them is isotopic to the Cayley table of the group $\mathbb{Z}_5$, and the other corresponds to the quasigroup $Q_5$.  The numbers of cuboctahedra of dimension $k$ are given in the table below. Hereinafter the last row of tables contains the lower bounds on the numbers of cuboctahedra   proved in Section~\ref{lowboundsec}.

$$
\begin{array}{|c||c|c|c|c|}
\hline
k & 0 & 1 & 2 & all \\
\hline
\hline
\mathbb{Z}_5 & 125 & 1000 & 2000 & 3125 \\
Q_5 & 125 & 1000 & 824 & 1949 \\
\hline
LB & 125 & 1000 & 680 & 1805 \\
\hline
\end{array}
$$

\textbf{Latin squares of order $6$.}
 There are  two  latin squares with  the maximum number of cuboctahedra   that are isotopic to the Cayley table of a group, namely, the groups $\mathbb{Z}_6$ or $S_3$.  The minimal number of cuboctahedra is achieved in all conjugates of the following  latin square:

$$
\begin{array}{cccccc}
0 & 1 & 2 & 3 & 4 & 5 \\
1 & 0 & 4 & 5 & 3 & 2 \\
2 & 3 & 5 & 1 & 0 & 4 \\
3 & 2 & 1 & 4 & 5 & 0 \\
4 & 5 & 3 & 0 & 2 & 1 \\
5 & 4 & 0 & 2 & 1 & 3
\end{array}
$$

The comparison of the  maximum, minimum, and lower bounds on the numbers of $k$-di\-men\-si\-onal cuboctahedra is given in the table below.

$$
\begin{array}{|c||c|c|c|c|}
\hline
k & 0 & 1 & 2 & all \\
\hline
\hline
\max & 216 & 2160 & 5400 & 7776 \\
\min  & 216 & 2160 & 1920 & 4296 \\
\hline
LB & 216 & 2160 & 1440 & 3816 \\
\hline
\end{array}
$$

 \textbf{Latin squares of order $7$.}
There is only one class of latin squares of order $7$ with the maximum number of cuboctahedra, namely the Cayley table of the group $\mathbb{Z}_7$. The minimum number of cuboctahedra is achieved at  six conjugates of  the following latin square:

$$
\begin{array}{ccccccc}
0 & 1 & 2 & 3 & 4 & 5 & 6 \\
1 & 2 & 4 & 6 & 5 & 3 & 0 \\
2 & 6 & 3 & 4 & 1 & 0 & 5 \\
3 & 0 & 6 & 5 & 2 & 4 & 1 \\
4 & 5 & 0 & 1 & 3 & 6 & 2 \\
5 & 4 & 1 & 0 & 6 & 2 & 3 \\
6 & 3 & 5 & 2 & 0 & 1 & 4
\end{array}
$$

The comparison of the  maximum, minimum, and the lower bounds on the  numbers of $k$-dimensional cuboctahedra in latin squares of order $7$ is given in the table below:

$$
\begin{array}{|c||c|c|c|c|}
\hline
k & 0 & 1 & 2 & all \\
\hline
\hline
\max & 343 & 4116 & 12348 & 16807 \\
\min  & 343 & 4116 & 3880 & 8339 \\
\hline
LB & 343 & 4116 & 2688 &  7147 \\
\hline
\end{array}
$$

Now let us study the numbers of cuboctahedra in latin hypercubes.
It is known that every \textbf{$d$-dimensional latin hypercube of order $3$} is isotopic to the iterated group $\mathbb{Z}_3$, so it has $3^{3d-1}$ cuboctahedra.

\textbf{$3$-Dimensional latin hypercubes of order $4$}.
We verified that any conjugate of a  given $3$-dimensional latin hypercube of order $4$  has the same number of cuboctahedra.  The maximum number of cuboctahedra occurs in three main classes of latin hypercubes: the iterated groups $\mathbb{Z}_4$ and $\mathbb{Z}_2^2$ and  a composition of $\mathbb{Z}_4$ and $\mathbb{Z}_2^2$.  The minimum number of cuboctahedra is achieved at the following latin hypercube:

$$
\begin{array}{cccc|cccc|cccc|cccc}
0 & 1 & 2 & 3 &  1 & 0 & 3 & 2 &  2 & 3 & 0 & 1 &  3 & 2 & 1 & 0 \\ 
1 & 0 & 3 & 2 &  0 & 1 & 2 & 3 &  3 & 2 & 1 & 0 &  2 & 3 & 0 & 1 \\ 
2 & 3 & 0 & 1 &  3 & 2 & 1 & 0 &  0 & 1 & 3 & 2 &  1 & 0 & 2 & 3 \\
3 & 2 & 1 & 0 &  2 & 3 & 0 & 1 &  1 & 0 & 2 & 3 &  0 & 1 & 3 & 2
\end{array}
$$

Let us see how the theoretical results of  Section~\ref{lowboundsec}  estimate  from below the number of cuboctahedra in $3$-dimensional latin hypercubes of order $4$. The numbers of $0$-dimensional and $1$-dimensional cuboctahedra are given by Propositions~\ref{0dimcub} and~\ref{1dimcub}, and the number of degenerate $k$-dimensional cuboctahedra is given by Proposition~\ref{degenercub}.  For the number of $2$-dimensional cuboctahedra,  the fact, that each of $12$  $2$-dimensional planes of a hypercube of order $4$ is isotopic to the Cayley table of a group and contains $4^5 = 1024$ cuboctahedra, implies a better lower bound than Propositions~\ref{degenercub} and~\ref{kdimnondeg}.

The following table, we  compare the maximum, minimum and the theoretical lower bounds on the numbers of $k$-dimensional cuboctahedra in $3$-dimensional latin hypercubes of order $4$:

$$
\begin{array}{|c||c|c|c|c|c|}
\hline
k & 0 & 1 & 2 & 3 & all \\
\hline
\hline
\max & 1024 & 9216 & 27648 & 27648 & 65536 \\
\min  & 1024 & 9216 & 21504 & 15360  & 47104\\
\hline
 LB & 1024 & 9216 & 12288 &  1728 & 24256 \\
 \hline
\end{array}
$$

\textbf{$3$-Dimensional latin hypercubes of order $5$.}   
Only  one latin $3$-dimensional  hypercube of order $5$ has the maximum number of cuboctahedra,  namely the iterated group $\mathbb{Z}_5$. The minimum number of cuboctahedra is achieved at all   conjugates of  the following latin hypercube:

$$
\begin{array}{ccccc|ccccc|ccccc|ccccc|ccccc}
0 & 1 & 2 & 3 & 4 &  1 & 0 & 3 & 4 & 2 &  2 & 3 & 4 & 0 & 1 &  3 & 4 & 1 & 2 & 0 &  4 & 2 & 0 & 1 & 3 \\
1 & 0 & 3 & 4 & 2 &  2 & 3 & 4 & 1 & 0 &  0 & 4 & 1 & 2 & 3 &  4 & 2 & 0 & 3 & 1 &  3 & 1 & 2 & 0 & 4 \\
2 & 4 & 1 & 0 & 3 &  4 & 2 & 0 & 3 & 1 &  3 & 0 & 2 & 1 & 4 &  0 & 1 & 3 & 4 & 2 &  1 & 3 & 4 & 2 & 0 \\
3 & 2 & 4 & 1 & 0 &  0 & 4 & 1 & 2 & 3 &  4 & 1 & 0 & 3 & 2 &  1 & 3 & 2 & 0 & 4 &  2 & 0 & 3 & 4 & 1 \\ 
4 & 3 & 0 & 2 & 1 &  3 & 1 & 2 & 0 & 4 &  1 & 2 & 3 & 4 & 0 &  2 & 0 & 4 & 1 & 3 &  0 & 4 & 1 & 3 & 2 
\end{array}
$$

Let us see how the theoretical reasoning of  Section~\ref{lowboundsec}  estimates  the numbers of cuboctahedra in $3$-dimensional latin hypercubes of order $5$.   Using again Propositions~\ref{0dimcub},~\ref{1dimcub}, and~\ref{degenercub},  we determine the numbers of  $0$-dimensional and $1$-dimensional cuboctahedra,  as well as the number of degenerate $k$-dimensional cuboctahedra.

To estimate the number of non-degenerate $2$-dimensional cuboctahedra, we refine  the proof of Proposition~\ref{kdimnondeg}. For this purpose, let us compare the number of submatrices with a given corner with the number of their possible fillings.  Without loss of generality, we may assume that the corner symbol is $0$. Then there are the following types of patterns of symbols in $2$-dimensional matrices of order $2$, where the corner of the submatrix is  always left-up   $0$,   different/same digits denote different/same symbols (not necessarily $1$ and $2$):

$$
(1) ~\begin{array}{cc}
0 & 1 \\ 1 & \cdot
\end{array} ~~~~
(2) ~ \begin{array}{cc}
0 & 1 \\ 2 & \cdot
\end{array}
$$
Let $N_i$ be the number of  different patterns of type $(i)$ and  $P_i$ be the number of ways to complete a pattern of type   $(i)$ to a latin submatrix of order $2$. Since for each direction of $2$-dimensional planes in a  $3$-dimensional latin hypercube of order $5$ there are exactly $25$ appearances of a given pattern, we get that each submatrix should appear at least   $R_i = \lfloor 25  / P_i  \rfloor$ times. In particular, for our patterns we have the following values:

$$
\begin{array}{c||c|c}
i & 1 & 2 \\
\hline
N_i & 4  & 12  \\
P_i & 4  & 3 \\
R_i & 6 &  8 \\
\end{array}
$$
So every $3$-dimensional latin hypercube of order $5$ has in addition 
$$  5  {3 \choose 2}  \left( \sum\limits_{i=1}^2 N_i P_i R_i(R_i - 1)  \right)   = 37440 $$  
non-degenerate $2$-dimensional cuboctahedra. 

It can be checked that every $3$-dimensional pattern in a $3$-dimensional latin hypercube of order $5$ can be completed in  many ways $P_i$  so that   the ratio  $R_i = \lfloor 25  / P_i  \rfloor$  is less than $2$. Thus, we cannot guarantee the existence of non-degenerate $3$-dimensional cuboctahedra. 

The comparison of the maximum and minimum numbers of $k$-dimensional cuboctahedra in $3$-dimensional latin hypercube of order $5$  with theoretical lower bounds is given in the table below.

$$
\begin{array}{|c||c|c|c|c|c|}
\hline
k & 0 & 1 & 2 & 3 & all \\
\hline
\hline
\max & 3125 & 37500 & 150000 & 200000 & 390625 \\
\min  & 3125 & 37500 & 53136 & 9984  & 103745\\
\hline
LB  & 3125 & 37500 & 43440 & 3375  & 87440\\
\hline
\end{array}
$$

\textbf{$4$-Dimensional latin hypercubes of order $4$.}
There are $8$  representatives from different main classes having the maximum number of cuboctahedra.  It is interesting to note that there is a $4$-dimensional latin hypercube of order $4$ with the maximum number of cuboctahedra, such that all its nontrivial conjugates have a non-maximal number of cuboctahedra. 

 The minimal number of cuboctahedra is achieved at the following latin hypercube:
 
 $$
 \begin{array}{cccc|cccc|cccc|cccc}
 0 &  1 &  2 &  3 &  1 &  0 &  3 &  2 &  2 &  3 &  1 &  0 & 3 &  2 &  0 &  1 \\   
 1 &  0 &  3 &  2  & 0 &  1 &  2 &  3 &   3 &  2 &  0 &  1 &   2 &  3 &  1 &  0 \\
 2 &  3 &  1 &  0 &   3 &  2 &  0 &  1 &   1 &  0 &  3 &  2 &   0 &  1 &  2 &  3 \\
 3 &  2 &  0 &  1 &   2 &  3 &  1 &  0 &   0 &  1 &  2 &  3 &   1 &  0 &  3 &  2 \\  
 \hline
  1 &  0 &  3 &  2 &   2 &  3 &  1 &  0 &   3 &  2 &  0 &  1 &   0 &  1 &  2 &  3 \\  
 2 &  3 &  1 &  0 &   1 &  0 &  3 &  2 &   0 &  1 &  2 &  3 &   3 &  2 &  0 &  1 \\  
 3 &  2 &  0 &  1 &   0 &  1 &  2 &  3 &   2 &  3 &  1 &  0 &   1 &  0 &  3 &  2 \\ 
 0 &  1 &  2 &  3 &   3 &  2 &  0 &  1 &   1 & 0 &  3 &  2 &   2 &  3 &  1 &  0 \\  
 \hline
  2 &  3 &  1 &  0 &   3 &  2 &  0 &  1 &   0 &  1 &  2 &  3 &   1 &  0 &  3 &  2 \\   
  3 &  2 &  0 &  1 &   2 &  3 &  1 &  0 &   1 &  0 &  3 &  2 &   0 &  1 &  2 &  3 \\  
  0 &  1 &  2 &  3 &   1 &  0 &  3 &  2 &   3 &  2 &  0 &  1 &   2 &  3 &  1 &  0 \\   
  1 &  0 &  3 &  2 &   0 &  1 &  2 &  3 &   2 &  3 &  1 &  0 &   3 &  2 &  0 &  1 \\  
  \hline
  3 &  2 &  0 &  1 &   0 &  1 &  2 &  3 &   1 &  0 &  3 &  2 &   2 &  3 &  1 &  0 \\   
  0 &  1 &  2 &  3 &   3 &  2 &  0 &  1 &   2 &  3 &  1 &  0 &   1 &  0 &  3 &  2 \\  
  1 &  0 &  3 &  2 &   2 &  3 &  1 &  0 &   0 &  1 &  2 &  3 &   3 &  2 &  0 &  1 \\  
  2 &  3 &  1 &  0 &   1 &  0 &  3 &  2 &   3 &  2 &  0 &  1 &   0 &  1 &  2 &  3
 \end{array}
 $$

Let us compute the theoretical lower bounds. Propositions~\ref{0dimcub},~\ref{1dimcub}, and~\ref{degenercub},  again give  the numbers of  $0$-dimensional and $1$-dimensional cuboctahedra  and  the number of degenerate $k$-dimensional cuboctahedra. 
Moreover, in this case the lower bound  on  $2$-dimensional non-degenerate cuboctahedra from Proposition~\ref{kdimnondeg} is better than one obtained from the group structure of $2$-dimensional planes. 

Indeed,  consider the  following types of patterns in $2$-dimensional matrices of order $2$ 
$$
(1) ~\begin{array}{cc}
0 & 1 \\ 1 & \cdot
\end{array} ~~~
(2) ~ \begin{array}{cc}
0 & 1 \\ 2 & \cdot
\end{array}
$$
 where the corner symbol is always $0$. Then the numbers   $N_i$ of  different patterns of type $(i)$,   the numbers $P_i$  of ways to complete a pattern of type   $(i)$ to a latin  submatrix of order $2$,  and   $R_i = \lfloor 64  / P_i  \rfloor$ are
$$
\begin{array}{c||c|c}
i & 1 & 2 \\
\hline
N_i & 3  & 6  \\
P_i & 3  & 2 \\
R_i & 21 &  32 \\
\end{array}
$$
So every $4$-dimensional latin hypercube of order $4$ has 
$$  4 \cdot  {4 \choose 2}  \left( \sum\limits_{i=1}^2 N_i P_i R_i(R_i - 1)  \right)   = 376416 $$  
non-degenerate $2$-dimensional cuboctahedra. 

Similarly, we can obtain an estimation for the number on non-degenerate $3$-dimensional cuboctahedra.
There are the following types of patterns in $3$-dimensional matrices of order $2$ 
$$
(1) ~\begin{array}{cc|cc}
0 & 1 & 1 & \cdot \\ 1 & \cdot & \cdot & \cdot
\end{array} ~~~
(2) ~\begin{array}{cc|cc}
0 & 1 & 2 & \cdot \\ 1 & \cdot & \cdot & \cdot
\end{array} ~~~
(3) ~\begin{array}{cc|cc}
0 & 1 & 3 & \cdot \\ 2 & \cdot & \cdot & \cdot
\end{array} ~~~
$$
The numbers $N_i$ of  different patterns of type $(i)$,   the numbers $P_i$  of ways to complete a pattern of type   $(i)$ to a latin submatrix of order $2$  and   $R_i = \lfloor 64  / P_i  \rfloor$ are
$$
\begin{array}{c||c|c|c}
i & 1 & 2 & 3 \\
\hline
N_i & 3  & 18 &  6  \\
P_i & 51  & 24  & 13  \\
R_i & 1 &  2 & 4  \\
\end{array}
$$

So every $4$-dimensional latin hypercube of order $4$ has 
$$  4 \cdot  {4 \choose 3}  \left( \sum\limits_{i=1}^3 N_i P_i R_i(R_i - 1)  \right)   = 28800 $$  
non-degenerate $3$-dimensional cuboctahedra. 

It can be checked, that every $4$-dimensional pattern in a $4$-dimensional latin hypercube of order $4$  can   be completed to a submatrix  in many ways $P_i$ so that   the ratio  $R_i = \lfloor 64  / P_i  \rfloor$  is less than $2$. Thus we cannot guarantee the existence of non-degenerate $4$-dimensional cuboctahedra. 

The comparison of the maximum and minimum numbers of $k$-dimensional  cuboctahedra with the theoretical lower bounds  is given in the table below.

$$
\begin{array}{|c||c|c|c|c|c|c|}
\hline
k & 0 & 1 & 2 & 3 & 4 & all \\
\hline
\hline
\max & 16384 & 196608 & 884736 & 1769572 & 1327104 & 4194304 \\
\min  & 16384 & 196608 & 698368 & 903168  & 405504 & 2220032 \\
\hline
LB  & 16384 & 196608 & 390240 &  56448 & 20736 & 680416 \\
\hline
\end{array}
$$

\textbf{$4$-Dimensional latin hypercube of order $5$.}
Up to isotopy, there is only one $4$-dimensional  latin hypercube of order $5$ containing   the maximum number of cuboctahedra, namely the iterated group $\mathbb{Z}_5$.  The minimum number of cuboctahedra is achieved at the following latin hypercube:

$$\begin{array}{ccccc|ccccc|ccccc|ccccc|ccccc}
0 & 1 & 2 & 3 & 4 & 1 & 0 & 3 & 4 & 2 & 4 & 2 & 0 & 1 & 3 & 2 & 3 & 4 & 0 & 1 & 3 & 4 & 1 & 2 & 0 \\ 
1 & 0 & 3 & 4 & 2 & 0 & 1 & 4 & 2 & 3 & 3 & 4 & 2 & 0 & 1 & 4 & 2 & 1 & 3 & 0 & 2 & 3 & 0 & 1 & 4 \\ 
4 & 2 & 0 & 1 & 3 & 3 & 4 & 2 & 0 & 1 & 2 & 3 & 1 & 4 & 0 & 1 & 0 & 3 & 2 & 4 & 0 & 1 & 4 & 3 & 2 \\ 
2 & 3 & 4 & 0 & 1 & 4 & 2 & 1 & 3 & 0 & 0 & 1 & 3 & 2 & 4 & 3 & 4 & 0 & 1 & 2 & 1 & 0 & 2 & 4 & 3 \\ 
3 & 4 & 1 & 2 & 0 & 2 & 3 & 0 & 1 & 4 & 1 & 0 & 4 & 3 & 2 & 0 & 1 & 2 & 4 & 3 & 4 & 2 & 3 & 0 & 1 \\ 
\hline
1 & 0 & 4 & 2 & 3 & 0 & 1 & 2 & 3 & 4 & 2 & 4 & 3 & 0 & 1 & 3 & 2 & 1 & 4 & 0 & 4 & 3 & 0 & 1 & 2 \\ 
0 & 1 & 2 & 3 & 4 & 1 & 0 & 3 & 4 & 2 & 4 & 3 & 1 & 2 & 0 & 2 & 4 & 0 & 1 & 3 & 3 & 2 & 4 & 0 & 1 \\ 
2 & 4 & 3 & 0 & 1 & 4 & 3 & 1 & 2 & 0 & 3 & 2 & 0 & 1 & 4 & 0 & 1 & 4 & 3 & 2 & 1 & 0 & 2 & 4 & 3 \\ 
3 & 2 & 1 & 4 & 0 & 2 & 4 & 0 & 1 & 3 & 1 & 0 & 4 & 3 & 2 & 4 & 3 & 2 & 0 & 1 & 0 & 1 & 3 & 2 & 4 \\ 
4 & 3 & 0 & 1 & 2 & 3 & 2 & 4 & 0 & 1 & 0 & 1 & 2 & 4 & 3 & 1 & 0 & 3 & 2 & 4 & 2 & 4 & 1 & 3 & 0 \\ 
\hline
4 & 3 & 0 & 1 & 2 & 2 & 4 & 1 & 0 & 3 & 3 & 1 & 2 & 4 & 0 & 1 & 0 & 3 & 2 & 4 & 0 & 2 & 4 & 3 & 1 \\ 
2 & 4 & 1 & 0 & 3 & 3 & 2 & 0 & 1 & 4 & 1 & 0 & 4 & 3 & 2 & 0 & 3 & 2 & 4 & 1 & 4 & 1 & 3 & 2 & 0 \\ 
3 & 1 & 2 & 4 & 0 & 1 & 0 & 4 & 3 & 2 & 0 & 4 & 3 & 2 & 1 & 4 & 2 & 1 & 0 & 3 & 2 & 3 & 0 & 1 & 4 \\
1 & 0 & 3 & 2 & 4 & 0 & 3 & 2 & 4 & 1 & 4 & 2 & 0 & 1 & 3 & 2 & 1 & 4 & 3 & 0 & 3 & 4 & 1 & 0 & 2 \\ 
0 & 2 & 4 & 3 & 1 & 4 & 1 & 3 & 2 & 0 & 2 & 3 & 1 & 0 & 4 & 3 & 4 & 0 & 1 & 2 & 1 & 0 & 2 & 4 & 3 \\ 
\hline
2 & 4 & 3 & 0 & 1 & 3 & 2 & 4 & 1 & 0 & 0 & 3 & 1 & 2 & 4 & 4 & 1 & 0 & 3 & 2 & 1 & 0 & 2 & 4 & 3 \\ 
3 & 2 & 4 & 1 & 0 & 4 & 3 & 2 & 0 & 1 & 2 & 1 & 0 & 4 & 3 & 1 & 0 & 3 & 2 & 4 & 0 & 4 & 1 & 3 & 2 \\ 
0 & 3 & 1 & 2 & 4 & 2 & 1 & 0 & 4 & 3 & 1 & 0 & 4 & 3 & 2 & 3 & 4 & 2 & 1 & 0 & 4 & 2 & 3 & 0 & 1 \\ 
4 & 1 & 0 & 3 & 2 & 1 & 0 & 3 & 2 & 4 & 3 & 4 & 2 & 0 & 1 & 0 & 2 & 1 & 4 & 3 & 2 & 3 & 4 & 1 & 0 \\ 
1 & 0 & 2 & 4 & 3 & 0 & 4 & 1 & 3 & 2 & 4 & 2 & 3 & 1 & 0 & 2 & 3 & 4 & 0 & 1 & 3 & 1 & 0 & 2 & 4 \\ 
\hline
3 & 2 & 1 & 4 & 0 & 4 & 3 & 0 & 2 & 1 & 1 & 0 & 4 & 3 & 2 & 0 & 4 & 2 & 1 & 3 & 2 & 1 & 3 & 0 & 4 \\ 
4 & 3 & 0 & 2 & 1 & 2 & 4 & 1 & 3 & 0 & 0 & 2 & 3 & 1 & 4 & 3 & 1 & 4 & 0 & 2 & 1 & 0 & 2 & 4 & 3 \\ 
1 & 0 & 4 & 3 & 2 & 0 & 2 & 3 & 1 & 4 & 4 & 1 & 2 & 0 & 3 & 2 & 3 & 0 & 4 & 1 & 3 & 4 & 1 & 2 & 0 \\ 
0 & 4 & 2 & 1 & 3 & 3 & 1 & 4 & 0 & 2 & 2 & 3 & 1 & 4 & 0 & 1 & 0 & 3 & 2 & 4 & 4 & 2 & 0 & 3 & 1 \\ 
2 & 1 & 3 & 0 & 4 & 1 & 0 & 2 & 4 & 3 & 3 & 4 & 0 & 2 & 1 & 4 & 2 & 1 & 3 & 0 & 0 & 3 & 4 & 1 & 2
\end{array}
$$

From Propositions~\ref{0dimcub},~\ref{1dimcub}, and~\ref{degenercub},  we find the numbers of  $0$-dimensional and $1$-dimensional cuboctahedra  and  the number of degenerate $k$-dimensional cuboctahedra.

Let us see how  Proposition~\ref{kdimnondeg} estimates the number of non-degenerate cuboctahedra in $4$-dimensional latin hypercubes of order $5$. 

As before, there are the  following types of patterns in $2$-dimensional  submatrices of order $2$ 
$$
(1) ~\begin{array}{cc}
0 & 1 \\ 1 & \cdot
\end{array} ~~~
(2) ~ \begin{array}{cc}
0 & 1 \\ 2 & \cdot
\end{array}
$$
 where the corner symbol is always $0$. The numbers   $N_i$ of  different patterns of type $(i)$,   the numbers $P_i$  of ways to complete a pattern of type   $(i)$ to a  submatrix of order $2$  and   $R_i = \lfloor 125  / P_i  \rfloor$ are
$$
\begin{array}{c||c|c}
i & 1 & 2 \\
\hline
N_i & 4  & 12  \\
P_i & 4  & 3 \\
R_i & 31 &  41 \\
\end{array}
$$
So every $4$-dimensional latin hypercube of order $5$ has 
$$  5 \cdot  {4 \choose 2}  \left( \sum\limits_{i=1}^2 N_i P_i R_i(R_i - 1)  \right)   = 2217600  $$  
non-degenerate $2$-dimensional cuboctahedra.

It can be checked, that are many ways to complete  $3$-dimensional and $4$-dimensional   patterns in a $4$-dimensional latin hypercube of order $5$, so we cannot guarantee the existence of non-degenerate  cuboctahedra of these dimensions.

In the following table, we  compare  of the maximum and minimum numbers of cuboctahedra in $4$-dimensional latin hypercubes of order $5$ with theoretical lower bounds.

$$
\begin{array}{|c||c|c|c|c|c|c|}
\hline
k & 0 & 1 & 2 & 3 & 4 & all \\
\hline
\hline
\max & 78125 & 1250000 & 7500000 & 20000000 & 20000000 & 48828125 \\
\min  & 78125 & 1250000 & 2551712 & 923216  & 267344 & 5080397 \\
\hline
LB  & 78125 & 1250000 & 2277600 & 160000  &  160000  &  3925725  \\
\hline
\end{array}
$$

\section*{Acknowledgements}

The  work is carried out within the framework of the state contract of the Sobolev Institute of Mathematics (project no. FWNF-2026-0011). The author is grateful to Kamil' Gareev for computation of the numbers of cuboctahedra in latin squares of orders $n \leq 7$.

\end{document}